\documentclass{proc-l}

\usepackage{amsmath,amsthm,amssymb,enumerate,mathrsfs,color}
\usepackage[breaklinks=true]{hyperref}

\theoremstyle{plain}

\newtheorem{theorem}[equation]{Theorem}
\newtheorem{cor}[equation]{Corollary}
\newtheorem{prop}[equation]{Proposition}
\newtheorem{lemma}[equation]{Lemma}
\newtheorem{utheorem}{\textrm{\textbf{Theorem}}}

\theoremstyle{definition}

\newtheorem{remark}[equation]{Remark}
\newtheorem{defn}[equation]{Definition}
\newtheorem{example}[equation]{Example}

\numberwithin{equation}{section}

\newcommand{\R}{\mathbb{R}}
\newcommand{\C}{\mathbb{C}}
\newcommand{\bp}{\mathbb{P}}
\newcommand{\calh}{\mathcal{H}}
\newcommand{\HS}{\mathcal{S}_2}
\newcommand{\tangle}[1]{\langle #1 \rangle}
\DeclareMathOperator{\Id}{\rm Id}
\DeclareMathOperator{\tr}{\rm tr}
\DeclareMathOperator{\rk}{\rm rk}

\begin{document}
\title[Sharp nonzero lower bounds for the Schur product theorem]{Sharp
nonzero lower bounds for the\\ Schur product theorem}

\author{Apoorva Khare}

\address{Department of Mathematics, Indian Institute of Science,
Bangalore -- 560012, India; and
Analysis \& Probability Research Group, Bangalore -- 560012, India}

\email{\tt khare@iisc.ac.in}

\date{\today}



\begin{abstract}
By a result of Schur [\textit{J.~reine angew.~Math.}~1911], the entrywise
product $M \circ N$ of two positive semidefinite matrices $M,N$ is again
positive.
Vyb\'iral [\textit{Adv.\ Math.}~2020] improved on this by showing the
uniform lower bound $M \circ \overline{M} \geq E_n / n$ for all $n \times
n$ real or complex correlation matrices $M$, where $E_n$ is the all-ones
matrix.
This was applied to settle a conjecture of Novak [\textit{J.~Complexity}
1999] and to positive definite functions on groups.
Vyb\'iral (in his original preprint) asked if one can obtain similar
uniform lower bounds for higher entrywise powers of $M$, or for $M \circ
N$ when $N \neq M, \overline{M}$. A natural third question is to ask for
a tighter lower bound that does not vanish as $n \to \infty$, i.e.~over
infinite-dimensional Hilbert spaces.

In this note, we affirmatively answer all three questions by extending
and refining Vyb\'iral's result to lower-bound $M \circ N$, for arbitrary
complex positive semidefinite matrices $M,N$.
Specifically: we provide tight lower bounds, improving on Vyb\'iral's
bounds. Second, our proof is `conceptual' (and self-contained), providing
a natural interpretation of these improved bounds via tracial
Cauchy--Schwarz inequalities. Third, we extend our tight lower bounds to
Hilbert--Schmidt operators.
As an application, we settle Open Problem~1 of
Hinrichs--Krieg--Novak--Vyb\'iral [\textit{J.\ Complexity}, in press],
which yields improvements in the error bounds in certain tensor product
(integration) problems.
\end{abstract}
\maketitle

\section{Introduction and main result}

%{{{1 Section 1.1 - The Schur product theorem and nonzero lower bounds}
\subsection{The Schur product theorem and nonzero lower bounds}

We begin with a few definitions. A \textit{positive semidefinite matrix}
is a complex Hermitian matrix with non-negative eigenvalues.
Denote the space of such $n \times n$ matrices by $\bp_n = \bp_n(\C)$.
Given integers $m,n \geq 1$, the \textit{Schur product}, or
\textit{entrywise product} of two (possibly rectangular) $m \times n$
complex matrices $A = (a_{jk}), B = (b_{jk})$ equals the $m \times n$
matrix $A \circ B$ with $(j,k)$ entry $a_{jk} b_{jk}$. 

A seminal result by Schur~\cite{Schur1911} asserts that if $M,N$ are
positive semidefinite matrices of the same size, then so is their
entrywise product $M \circ N$. This fundamental observation has had
numerous follow-ups and applications; perhaps the most relevant to the
present short note is the development of the entrywise calculus in matrix
analysis, with connections to numerous classical and modern works, both
theoretical and applied. (See e.g.~the two-part
survey~\cite{BGKP-survey1,BGKP-survey2}.) It also extends to positive
self-adjoint operators on Hilbert spaces.

The Schur product theorem is often phrased using the \textit{Loewner
ordering} on $\C^{n \times n}$ -- in which $M \geq N$ if $M - N \in
\bp_n = \bp_n(\C)$ -- in the following form:
\begin{equation}
M \geq {\bf 0}_{n \times n}, \ N \geq {\bf 0}_{n \times n}
\quad \implies \quad M \circ N \geq {\bf 0}_{n \times n}.
\end{equation}
This is a `qualitative' result, in that it provides a lower bound of
${\bf 0}_{n \times n}$ for $M \circ N$ for all $M,N \in \bp_n$. It is
natural to seek `quantitative' results, i.e., nonzero lower bounds. Here
are some known bounds: Fiedler's inequality~\cite{F} says $A \circ A^{-1}
\geq \Id_n$ if $A \in \bp_n$ is invertible. Two more examples, see
e.g.~\cite{FM,R}, are:
\begin{align}
\begin{aligned}
M \circ N \geq &\ \lambda_{\min}(N) (M \circ \Id_n), \ \textit{if } M \in
\bp_n \text{ is real and } N = N^T \in \R^{n \times n},\\
M \circ N \geq &\ \frac{1}{{\bf e}^T N^{-1} {\bf e}} M, \ \text{if }
M,N \in \bp_n \text{ and } \det(N) > 0.
\end{aligned}
\end{align}
Here and below, we use the following \textbf{notation} without further
reference.
\begin{itemize}
\item Given a fixed integer $n \geq 1$, let ${\bf e} = {\bf e}(n) :=
(1,\dots,1)^T \in \C^n$, and $E_n := {\bf e} {\bf e}^T \in \bp_n$.

\item We say that a matrix in $\bp_n$ is a \textit{real/complex
correlation matrix} if it has all diagonal entries $1$, and all entries
real/complex respectively.

\item Given a matrix $M_{n \times n}$ and a subset $J \subset \{ 1,
\dots, n \}$, let $M_{J \times J}$ denote the principal submatrix of $M$
corresponding to the rows and columns indexed by $J$; and let $d_M :=
(m_{11}, \dots, m_{nn})^T$.
\end{itemize}

This note concerns the recent paper~\cite{V}, in which Vyb\'iral showed a
new lower bound for all $M \circ \overline{M}$, where $M$ is a
correlation matrix:

\begin{theorem}[\cite{V}]\label{TV1}
If $n \geq 1$ and $M_{n \times n}$ is a real or complex correlation
matrix (so $\overline{M} = M^T$), then $M \circ \overline{M} \geq
\frac{1}{n} E_n$.
\end{theorem}

Theorem~\ref{TV1} is striking in its simplicity (and in that it seems to
have been undiscovered for more than a century after the Schur product
theorem~\cite{Schur1911}). There are no obvious upper bounds for the
left-hand side, while it is \textit{a priori} intriguing that there is a
nonzero lower bound.

Vyb\'iral provided a direct proof, in fact of a more general fact:

\begin{theorem}[\cite{V}]\label{TV2}
Given a matrix $M \in \C^{n \times n}$, let $d_M := (m_{11}, \dots,
m_{nn})^T$ be the vector consisting of its diagonal entries. Now if $M
\in \bp_n$, then $M \circ \overline{M} \geq \frac{1}{n} d_M d_M^T$.
\end{theorem}

Vyb\'iral used these results to prove a conjecture of Novak~\cite{N} in
numerical integration (see Theorem~\ref{Tnovak}), with applications to
positive definite functions and in other areas. See~\cite{V} for details.
%}}}

%{{{1 Section 1.2 - The main result
\subsection{The main result}

Following the above results, Vyb\'iral asked -- at the end of his
original 2019 preprint~\cite{V19} -- if Theorem~\ref{TV1} admits variants
(1)~for $M \circ N$ for $N \neq M, \overline{M}$; and
(2)~for higher powers of $M$.
He answered~(1) in his updated paper, as follows:

\begin{theorem}[\cite{V}]\label{TV3}\hfill
\begin{enumerate}
\item If $M = A A^*, N = B B^* \in \bp_n$ for $A,B \in \C^{n \times n}$,
then $M \circ N \geq \frac{1}{n} w w^*$, with $w := (A \circ B) {\bf e}$.
\item In particular, setting $B = {\bf e} {\bf e}^T = E_n$, we have $A
A^* \geq \frac{1}{n} (A {\bf e}) (A {\bf e})^*$.
\end{enumerate}
\end{theorem}

\noindent Note, the first part implies Theorem~\ref{TV2} (whence
Theorem~\ref{TV1}) by setting $B = \overline{A}$. Thus,
Theorem~\ref{TV3}(1) is currently state-of-the-art.

The lower bound of $1/n$ poses a technical challenge to the functional
analyst: Theorem~\ref{TV3} cannot be extended to infinite-dimensional
Hilbert spaces to yield a nontrivial lower bound. It is thus natural to
ask (3)~whether there exists a function of $M,N$ (or of $A,B$) that can
improve the constant $1/n$ to a bound that remains nonzero in Hilbert
spaces.\medskip

The contributions of this short note are as follows:
\begin{itemize}
\item Our main result indeed provides an improved bound sought-for above,
so that it also extends to a nonzero lower bound in the Hilbert space
setting (see Section~\ref{Shilbert}).

\item We show this improved bound is tight, and strictly improves on the
state-of-the-art Theorem~\ref{TV3}. We also do not require $A,B$ to be
square matrices -- or even equi-dimensional.

\item The proof we provide is conceptual and `coordinate-free', in
contrast to previous direct and `computational' proofs of special cases.
(At the same time, our proof uses elementary arguments, whence is
self-contained.) In particular, we show that the results here and by
Vyb\'iral are all tracial Cauchy--Schwarz inequalities -- our proof also
explains the \textit{meaning} of our tight bound.
\end{itemize}

Here is the main result of this note.

\begin{utheorem}\label{Tmain0}
Given integers $n, a \geq 1$ and nonzero matrices $A,B \in \C^{n \times
a}$, we have the (rank $\leq 1$) lower bound:
\begin{equation}\label{Eineq0}
A A^* \circ B B^* \geq \ \frac{1}{\min(\rk(A A^*), \rk(B B^*))} \cdot
d_{A B^T} d_{A B^T}^*,
\end{equation}
and the choice of constant is best possible.
\end{utheorem}

\noindent (In fact we do not require $A,B$ to have the same number of
columns; see Corollary~\ref{Cmain0} below.) Before proving this theorem,
we discuss some special cases, beginning with the solution to an open
problem.

Theorem~\ref{Tmain0} finds an application in numerical integration (in
the spirit of Vyb\'iral's original result~\cite{V} being recently applied
to resolve Novak's conjecture~\cite{N}). Specifically, in the recent work
by Hinrichs--Krieg--Novak--Vyb\'iral~\cite{HKNV}, the authors prove two
results (Theorems~$15$ and~$16$ in \textit{loc.\ cit.}); the latter
states that given integers $n, D \geq 1$, and real matrices $A, B \in
\R^{n \times D}$ with $A A^T = B B^T$ of rank $r > 0$, we have:
\begin{equation}
A A^T \circ B B^T \geq \frac{1}{2r} d_{A B^T} d_{A B^T}^T.
\end{equation}
The authors then ask (see Open Problem~$1$ in~\cite{HKNV}) if the
constant $1/(2r)$ can be improved to $1/r$; this would lead to improved
error bounds in certain tensor product integration problems. This Open
Problem -- as well as both of their aforementioned theorems -- are
immediate consequences of Theorem~\ref{Tmain0}. For instance, in the
special case $A A^T = B B^T$, Theorem~\ref{Tmain0} above answers the Open
Problem (in particular, improving on~\cite[Theorem 16]{HKNV}):

\begin{cor}
Given arbitrary integers $n,D \geq 1$ and nonzero matrices $A,B \in \R^{n
\times D}$, if $A A^T = B B^T = M$ then
\[
A A^T \circ B B^T \geq \frac{1}{\rk(M)} d_{A B^T} d_{A B^T}^T.
\]
\end{cor}

We next discuss additional special cases of our main result, which were
previously proved in the literature.

\begin{remark}[Specializing to earlier results]
Theorem~\ref{Tmain0} extends and unifies the preceding results above. It
suffices to deduce the `state-of-the-art' Theorem~\ref{TV3}(1). Letting
$v_j, w_k$ denote the columns of $A_{n \times n}, B_{n \times n}$
respectively, we have $A A^* = \sum_j v_j v_j^*$ and $B B^* = \sum_k w_k
w_k^*$. Hence
\[
A A^* \circ B B^* = \sum_{j,k = 1}^n (v_j v_j^*) \circ (w_k w_k^*) \geq
\sum_{j=1}^n (v_j \circ w_j) (v_j \circ w_j)^* = (A \circ B) (A \circ
B)^*.
\]
This shows that the two assertions in Theorem~\ref{TV3} are equivalent;
and setting $a=n, B = E_n / \sqrt{n}$ in Theorem~\ref{Tmain0} yields
Theorem~\ref{TV3}(2).
\end{remark}

\begin{remark}
Another special case that Vyb\'iral has separately communicated to
us~\cite{V20}, again holds for square matrix decompositions:
\begin{equation}\label{Eineq4}
A A^* \circ B B^* \geq \frac{1}{\max(\rk(A A^*), \rk(B B^*))} d_{A B^T}
d^*_{A B^T}, \qquad \forall A,B \in \C^{n \times n}.
\end{equation}
More precisely, Vyb\'iral mentioned that given any two positive matrices
$M,N \in \bp_n$, one has the lower bound~\eqref{Eineq4} for every pair of
decompositions $M = A A^*, N = B B^*$ for square matrices $A,B \in \C^{n
\times n}$. Notice that:
(a)~this holds only in the special case $a=n$ of Theorem~\ref{Tmain0};
(b)~the bound in~\eqref{Eineq4} is also not tight, as the coefficient of
$1/\max$ can be improved to $1/\min$ in
Theorem~\ref{Tmain0}; and
(c)~it is not clear if this statement implies Theorem~\ref{TV3}, or
conversely. Our main result, Theorem~\ref{Tmain0}, clearly unifies and
strengthens all of these variants.
\end{remark}

Having discussed the myriad special cases of the theorem, here is a proof
(that is self-contained on the one hand, and on the other, explains the
tight lower bound):

\begin{proof}[Proof of Theorem~\ref{Tmain0}]
The key identity needed to prove~\eqref{Eineq0} is algebraic: given any
square $n \times n$ matrices $M,N$ and vectors $u,v$ with $n$ coordinates
(over a unital commutative ring),
\begin{equation}\label{Eidentity0}
u^T (M \circ N) v = \tr ( N^T D_u M D_v ),
\end{equation}
where $D_u$ for a vector $u \in \C^n$ is the diagonal matrix with $(j,j)$
entry $u_j$.
Thus, pre- and post-multiplying the left-hand side of~\eqref{Eineq0} by
$u^*, u$ respectively, we compute:
\[
u^* (A A^* \circ B B^*) u = \tr( \overline{B} B^T D_{\overline{u}} A A^*
D_u ) = \tr( N^* N), \quad \text{where} \quad N := A^* D_u \overline{B}.
\]

Consider the inner product on $\C^{a \times a}$, given by $\tangle{X,Y}
:= \tr(X^* Y)$, and define the projection
\begin{equation}\label{Eproj}
P := {\rm proj}_{(\ker A)^\perp}|_{{\rm im}(B^T)};
\end{equation}
thus $P \in \C^{a \times a}$. We compute:
\begin{align*}
\tangle{P,P} \leq \min(\dim (\ker A)^\perp, \dim {\rm im}(B^T)) = &\
\min(\rk(A^*), \rk(B^*))\\
= &\ \min(\rk(A A^*), \rk(B B^*)).
\end{align*}
Hence by the Cauchy--Schwarz inequality (for this tracial inner
product),
\begin{align*}
u^* (A A^* \circ B B^*) u = \tangle{N,N} \geq &\
\frac{|\tangle{N,P}|^2}{\tangle{P,P}} = \frac{|\tr( A P B^T
D_{\overline{u}})|^2}{\tangle{P,P}} = \frac{|u^* d_{A P
B^T}|^2}{\tangle{P,P}}\\
\geq &\ \frac{1}{\min(\rk(A A^*), \rk(B B^*))} u^* d_{A P B^T} d^*_{A P
B^T} u.
\end{align*}
But this holds for all vectors $u$. This shows~\eqref{Eineq0} where $d_{A
B^T}$ is replaced by $d_{A P B^T}$; but in fact $A P B^T = A B^T$ by
choice of $P$.

Finally, we show the tightness of the bound $1 / \min( \rk(A A^*), \rk(B
B^*))$ (e.g. over $1/\max$). Choose integers $1 \leq r$ with $r,s \leq
n$, and complex block diagonal matrices
\[
A_{n \times a} := \begin{pmatrix} D_{r \times r} & 0 \\ 0 & 0
\end{pmatrix}, \quad
B_{n \times a} := \begin{pmatrix} D'_{s \times s} & 0 \\ 0 & 0
\end{pmatrix},
\]
with both $D, D'$ nonsingular. Then $P := \begin{pmatrix} \Id_{\min(r,s)}
& 0 \\ 0 & 0 \end{pmatrix}$, and the bound of $1 / \min(r,s)$ is indeed
tight, as can be verified using the Cauchy--Schwarz identity.
\end{proof}

We end this part with additional remarks, beginning by attaining equality
in \eqref{Eineq0}.

\begin{example}\label{Erankone}
Suppose $A = u, B = v$ are nonzero vectors in $\C^n$. Then~\eqref{Eineq0}
says:
\[
u u^* \circ v v^* \geq d_{u v^T} d^*_{u v^T} = (u \circ v) (u \circ v)^*.
\]
Thus, the inequality~\eqref{Eineq0} reduces to an equality for rank-one
matrices $A A^*, B B^*$.
\end{example}

\begin{remark}
Another way to consider Theorem~\ref{Tmain0} is to start with matrices
$M, N \in \bp_n(\C)$ and then obtain the bound~\eqref{Eineq0} for every
decomposition $M = A A^*, N = B B^*$. In this case, it is clear that the
constant 
\[
\gamma = \min(\rk M_{J \times J}, \rk N_{J \times J})^{-1}
\]
does not change; but the rank-one lower bound can indeed change. Even if
one runs over decompositions in terms of \textit{square} matrices $A,B$
(to dispense with the role of $P$), and assumes $C = \Id_J$, it would be
interesting to obtain some understanding of the possible rank-one
matrices obtained as lower bounds.

This is also linked to the possibility of obtaining higher-rank lower
bounds for $A A^* \circ B B^*$. One way to do so is to realize that the
left-hand side of~\eqref{Eineq0} is bi-additive in $(A A^*, B B^*)$, so
one can decompose both $A A^*$ and $B B^*$ as sums of lower-rank matrices
and obtain rank-one lower bounds for each pair of lower-rank matrices.
Example~\ref{Erankone} is relevant here: it shows that if one writes $A
A^*, B B^*$ as sums of rank-one matrices, then each corresponding
inequality is an equality, and adding these yields the unique best lower
bound of $A A^* \circ B B^*$.
\end{remark}
%}}}

%{{{1 Section 1.3 - Refinements using coordinates
\subsection{Refinements using coordinates}

We now present several refinements of Theorem~\ref{Tmain0}.
The first is \textit{a priori} more general, but in fact equivalent:

\begin{cor}\label{Cmain0}
Given integers $n, a, b \geq 1$ and nonzero matrices $A \in \C^{n \times
a}, \ B \in \C^{n \times b}$, we have the (rank $\leq 1$) lower bound:
\begin{equation}
A A^* \circ B B^* \geq \ \frac{1}{\min(\rk(A A^*), \rk(B B^*))} \cdot
d_{A_0 B_0^T} d_{A_0 B_0^T}^*,
\end{equation}
where $A_0$ appends $p + \max(a,b) - a$ zero-columns to the right of $A$,
and $B_0$ appends $p + \max(a,b) - b$ zero-columns to the right of $B$,
for some integer $p \geq 0$. Moreover, the choice of constant is best
possible.
\end{cor}

\begin{proof}
Clearly this result implies Theorem~\ref{Tmain0} by setting $b=a$ and
$p=0$. Conversely, apply Theorem~\ref{Tmain0} to $A_0, B_0$, and use $p +
\max(a,b)$ in place of $a$, to obtain:
\[
A_0 A_0^* \circ B_0 B_0^* \geq \ \frac{1}{\min(\rk(A_0 A_0^*), \rk(B_0
B_0^*))} \cdot d_{A_0 B_0^T} d_{A_0 B_0^T}^*.
\]
Now notice that $A_0 A_0^* = A A^*$ and $B_0 B_0^* = B B^*$.
\end{proof}

The next result refines Theorem~\ref{Tmain0} in the following sense:
suppose the matrix $M$ has nonzero entries only in the $J \times J$
coordinates (for a nonempty subset $J \subset \{ 1, \dots, n \}$). Then
the bound can in fact be improved:

\begin{theorem}\label{Tmain2}
Given integers $n, k \geq 1$ and a complex matrix $C_{k \times n}$, let
$J \subset \{ 1, \dots, n \}$ index the nonzero columns of $C$. Then for
all integers $a \geq 1$ and matrices $A \in \C^{n \times a}, \ B \in
\C^{n \times a}$ such that $(A A^*)_{J \times J}, (B B^*)_{J \times J}$
are nonzero, we have the (rank $\leq 1$) lower bound:
\begin{equation}\label{Eineq3}
C (A A^* \circ B B^*) C^* \ \geq \ \gamma(A,B,J) \cdot C d_{A B^T} d_{A
B^T}^* C^*,
\end{equation}
where the following choice of scalar $\gamma(A,B,J)$ is best possible:
\begin{equation}\label{Egamma}
\gamma(A,B,J) := \frac{1}{\min(\rk(A A^*)_{J \times J}, \rk(B B^*)_{J
\times J})}.
\end{equation}
\end{theorem}

\noindent Clearly, this implies Theorem~\ref{Tmain0} by setting $k=n$ and
$C = {\rm Id}_n$, so that $J = \{ 1, \dots, n \}$. However, it is
essentially also implied by it, as the following proof reveals.

\begin{proof}
A preliminary observation is that if $(A A^*)_{J \times J} = 0$ then
$\Id_J (A A^*) \Id_J = 0$, where $\Id_J \in \bp_n$ has diagonal entries
${\bf 1}_{i \in J}$. But then the submatrix $A_{J \times \{ 1, \dots, a
\}} = 0$, whence
\[
C d_{A B^T} = C \Id_J d_{A B^T} = 0.
\]
Thus the matrices on both sides of~\eqref{Eineq3} are zero, and so the
coefficient is irrelevant. The same conclusion is obtained by a similar
argument if $(B B^*)_{J \times J} = 0$.

We now prove~\eqref{Eineq3}. First observe that $C = C \Id_J$, so
that~\eqref{Eineq3} for $(C,A,B)$ follows from~\eqref{Eineq3} for $(C =
\Id_J, A, B)$. But this is precisely~\eqref{Eineq3} for the matrices
$(\Id_J, \Id_J A, \Id_J B)$. In other words, by restricting to the $J
\times J$ principal submatrices on both sides, we may assume without loss
of generality that $J = \{ 1, \dots, n \}$ and $C = \Id_n$; the
hypotheses imply $A,B$ are nonzero. This is precisely
Theorem~\ref{Tmain0}.
\end{proof}

We conclude this section by observing that~\eqref{Eineq3} can be extended
to Schur products of any number of positive matrices. Here are two sample
results:

\begin{cor}
Let $m,n,l \geq 1$ and matrices $M_1, \dots, M_m \in \bp_n$. Given a
partition of $\{ 1, \dots, m \}$ into subsets $J_1 \sqcup \cdots \sqcup
J_{2l}$, let
\[
M'_j := \circ_{i \in J_j} M_i, \qquad 1 \leq j \leq 2l.
\]
Now if $M'_j = A_j A_j^*$ for all $j \geq 1$, with each $A_j$ square and
nonzero, then we have the (rank $\leq 1$) lower bound:
\[
M'_1 \circ \cdots \circ M'_{2l} \geq \prod_{j=1}^k
\frac{1}{\min(\rk(M'_j), \rk(M'_{j+l}))} {\bf w} {\bf w}^*, \qquad
\text{where } {\bf w} := \circ_{j=1}^l d_{A_j A^T_{j+l}}.
\]
\end{cor}

While this result implies~\eqref{Eineq3} for $l=1$, $n=k$, and $C =
\Id_n$, it is also implied by it, via the `monotonicity' of the Schur
product: if $A \geq B$ and $A' \geq B'$, then $A \circ A' \geq B \circ A'
\geq B \circ B'$.

\begin{theorem}\label{Tmain}
Given a vector $u = (u_1, \dots, u_n)^T \in \C^n$, let $D_u$ denote the
diagonal matrix whose diagonal entries are the coordinates $u_1, \dots,
u_n$ of $u$; and let $J(u) \subset \{ 1, \dots, n \}$ denote the nonzero
coordinates of $u$, i.e.~$\{ j : 1 \leq j \leq n, \ u_j \neq 0 \}$.

Now let $k \geq 1$, and fix vectors $u_1, y_1, \dots, u_k, y_k \in \C^n$
such that ${\bf w} := (u_1 \circ y_1) \circ \cdots \circ (u_k \circ y_k)$
is nonzero. Then we have the (rank $\leq 1$) lower bound:
for all matrices $M_1, \dots, M_k \in \bp_n$,
\begin{align}\label{Eineq}
\begin{aligned}
&\ \left( D_{u_1} M_1 D^*_{u_1} \circ D_{y_1} \overline{M}_1 D^*_{y_1}
\right) \circ \cdots \circ \left( D_{u_k} M_k D^*_{u_k} \circ D_{y_k}
\overline{M}_k D^*_{y_k} \right)\\
\geq &\ \frac{1}{\rk(M_{J({\bf w}) \times J({\bf w})})}
({\bf w} \circ d_{M_1} \circ \cdots \circ d_{M_k}) 
({\bf w} \circ d_{M_1} \circ \cdots \circ d_{M_k})^*,
\end{aligned}
\end{align}
where $M := M_1 \circ \cdots \circ M_k$. Note, if the principal submatrix
$M_{J({\bf w}) \times J({\bf w})} = {\bf 0}$ then ${\bf w} \circ d_{M_1}
\circ \cdots \circ d_{M_k}$ is also zero, so the coefficient is
irrelevant.

Moreover, the coefficient $\frac{1}{\rk(M_{J({\bf w}) \times J({\bf
w})})}$ is best possible for all $u_j, y_j \in \C^n$ for which ${\bf w}
\neq 0$, and all $M_1, \dots, M_k$ for which $M_{J({\bf w}) \times J({\bf
w})} \neq {\bf 0}$.
\end{theorem}

Theorem~\ref{Tmain} is a tighter refinement of the Schur product theorem
than Theorem~\ref{TV2}, which is the special case with $k = 1$ and $u_1 =
y_1 = {\bf e}$. Moreover, Theorem~\ref{Tmain} can (and does) extend to
provide nonzero lower bounds in infinite-dimensional Hilbert spaces,
unlike Theorems~\ref{TV1} and~\ref{TV2}. We leave the proof to the
interested reader, as it is similar to (and follows from) theorems above.

\begin{remark}\label{Rlevel}
Define for a nonzero vector $d \in \C^n$, the `level set'
\[
\mathcal{S}_d := \{ (M_1, \dots, M_k) : \ M_j \in \bp_n\ \forall j, \
d_{M_1} \circ \cdots \circ d_{M_n} = d \}.
\]
Then a consequence of Theorem~\ref{Tmain} for $u_j = y_j = {\bf e}\
\forall j$, is that~\eqref{Eineq} provides a uniform lower bound on each
set $\mathcal{S}_d$ (i.e., which depends only on $d$). In fact the case
of $M$ a correlation matrix in~\cite{V}, is a special case of this
consequence for $d = {\bf e}$ (and $k=1$).
\end{remark}

We conclude with a `negative' remark, which shows that one cannot deviate
very far from the above hypotheses on the matrices in question.

\begin{remark}
Given the above results, a natural question is if even the original
identity $M \circ \overline{M} \geq \frac{1}{n} d_M d_M^T$ of Vyb\'iral
holds more widely. A natural extension to explore is from matrices $M
\circ \overline{M}$ to the larger class of \textit{doubly non-negative
matrices}: namely, matrices in $\bp_n$ with non-negative entries. In
other words, given a doubly non-negative matrix $A \in \bp_n$, is it true
that
\[
A \geq \frac{1}{n} d_A^{\circ 1/2} (d_A^{\circ 1/2})^T, \quad \text{where
} d_A^{\circ 1/2} := (a_{11}^{1/2}, \dots, a_{nn}^{1/2})^T?
\]

While this question was not addressed in~\cite{V}, it is easy to verify
that it is indeed true for $2 \times 2$ matrices. However, here is a
family of counterexamples for $n=3$; we leave the case of higher values
of $n$ to the interested reader. Consider the real matrix
\[
A = \begin{pmatrix} a & c & d \\ c & b & c \\ d & c & a \end{pmatrix},
\quad \text{where } a,b > 0, \ c \in [\sqrt{ab/2}, \sqrt{ab}), \ \ d =
\frac{2c^2}{b} - a < a.
\]
These bounds imply $A$ is doubly non-negative. Now we compute:
\[
A - d_A^{\circ 1/2} (d_A^{\circ 1/2})^T = \frac{1}{3} \begin{pmatrix}
2a & 3c - \sqrt{ab} & 3d - a\\
3c - \sqrt{ab} & 2b & 3c - \sqrt{ab}\\
3d - a & 3c - \sqrt{ab} & 2a
\end{pmatrix}.
\]
Straightforward computations show that all entries and $2 \times 2$
principal minors of this matrix are non-negative; but its determinant
equals
\[
\frac{2}{3} (a-d) (2 \sqrt{ab} c + bd - 3c^2) =
\frac{-2}{3} (a-d) (\sqrt{ab} - c)^2 < 0.
\]
This shows that one cannot hope to go much beyond the above test-set of
matrices $M \circ \overline{M}$, along the lines of the lower bound
in~\eqref{Eineq}.
\end{remark}
%}}}

%{{{1 Section 1.4 - An upper bound
\subsection{An upper bound}

While an upper bound on $M \circ N$ is not the focus of the present
paper, we provide one for completeness. The following statement depends
separately on $M,N$, not using $M \circ N$:

\begin{prop}
Given matrices $M,N \in \bp_n(\C)$, let $J \subset \{ 1, \dots, n \}$
comprise the indices $j$ such that $m_{jj}, n_{jj} > 0$,
and let $D_M, D_N$ denote the diagonal matrices $M \circ \Id_n, N \circ
\Id_n$ respectively. Also suppose $C_J(M)$ denotes the $J \times J$
`correlation' matrix with $(j,k)$ entry $m_{jk} / \sqrt{m_{jj} m_{kk}}$,
and similarly for $C_J(N)$.
Then,
%\[
%M \circ N \leq \alpha D_M D_N, \quad \text{where} \ \alpha = 1 + (|J|-1)
%\max_{j \neq k \in J} \frac{|m_{jk} n_{jk}|}{\sqrt{m_{jj} m_{kk} n_{jj}
%n_{kk}} }.
%\]
%Moreover,
\[
M \circ N \leq \max_{j \in J} (\| C_J(M)_{\ast j} \| \cdot \|
C_J(N)_{\ast j} \|) \cdot D_M D_N,
\]
where $C_{\ast j}$ for a matrix $C \in \C^{J \times J}$ denotes its $j$th
column.
\end{prop}

Note that this bound is indeed attained. In fact when $M,N$ are diagonal
matrices, we obtain an \textit{equality} of matrices.

\begin{proof}
First note that the matrices $M \circ N$ and $D_M D_N$ have nonzero
entries only in the $J \times J$ locations. Thus we may assume $J = \{ 1,
\dots, n \}$ without loss of generality.
Next, $M_{J \times J} = \sqrt{D_M} C_J(M) \sqrt{D_M}$, and similarly for
$N = N_{J \times J}$. Thus, if one shows the result with $M,N$ replaced
by $C_J(M), C_J(N)$ respectively (in which case $D_M, D_N$ are replaced
by $\Id_n$), then the general result follows. Thus, we assume henceforth
that $J = \{1, \dots, n \}$ and $M,N$ have all diagonal entries $1$.
Now $M \circ N \leq \lambda_{\max}(M \circ N) \Id_n$ by the spectral
theorem, where $\lambda_{\max}(\cdot)$ denotes the largest eigenvalue.
But this yields
\[
\lambda_{\max}(M \circ N) \leq \max_{j \in J} \sum_{k \in J} |m_{jk}
n_{jk}| \leq \max_{j \in J} \| M^J_{\ast j} \| \cdot \| N^J_{\ast j} \|
\]
by Gershgorin's circle theorem and the Cauchy--Schwarz inequality. 
\end{proof}
%}}}

%{{{1 Section 2 - Extension to Hilbert spaces
\section{Extension to Hilbert spaces}\label{Shilbert}

As mentioned in the discussion preceding Theorem~\ref{Tmain0}, we now
extend that result to Hilbert spaces. Let $(\calh, \tangle{\cdot,\cdot})$
be a real or complex Hilbert space with a \textit{fixed} orthonormal
basis $\{ e_x : x \in X \}$ -- so its span is dense in $\calh$. We begin
by recalling a few well-known notions, both basic and more advanced. In
what follows, $A, B : \calh \to \calh'$ are linear maps, with $\calh'$
another Hilbert space with orthonormal basis $\{ f_y : y \in Y \}$:
\begin{enumerate}
\item The \textit{adjoint} $A^* : \calh' \to \calh$ of $A$ is given by:
$\tangle{A^* f_y, e_x} := \tangle{f_y, A e_x}$ for all $x \in X, y \in
Y$. We will also freely use $u^*$ for a vector $u \in \calh$ to denote
the linear functional $\tangle{u,\cdot}$.

\item The \textit{transpose} of $A$ is $A^T : \calh' \to \calh$, given
by:
$\tangle{A^T f_y, e_x} := \tangle{A e_x, f_y}$ for $x \in X, y \in Y$.
The \textit{conjugate} $\overline{A} : \calh \to \calh'$ is precisely
$(A^*)^T = (A^T)^*$, given by $\tangle{f_y, \overline{A} e_x} :=
\tangle{A e_x, f_y}$.

\item The \textit{Schur product} of $A,B$ is the operator $A \circ B$
determined by:
$\tangle{f_y, (A \circ B) e_x} := \tangle{f_y, A e_x} \tangle{f_y, B
e_x}$ for all $x \in X, y \in Y$.

\item We say $A$ is \textit{bounded} if $A$ maps bounded sets into
bounded sets. Denote the collection of such bounded linear maps by
$\mathcal{B}(\calh, \calh')$, and by $\mathcal{B}(\calh)$ if $\calh' =
\calh$. The operator norm of $A \in \mathcal{B}(\calh, \calh')$ is $\| A
\| := \sup \{ \| Ax \|_{\calh'} : \| x \|_{\calh} \leq 1 \}$.

\item We say $A$ is \textit{Hilbert--Schmidt} if its Hilbert--Schmidt /
Frobenius norm is finite:
\[
\sum_{x \in X} \| A e_x \|^2 < \infty.
\]
Denote the set of Hilbert--Schmidt operators by $\HS(\calh, \calh')$ (the
Schatten 2-class), and by $\HS(\calh)$ if $\calh' = \calh$.

\item For $\calh = \calh'$, a Hilbert--Schmidt operator $A : \calh \to
\calh$ is \textit{trace class} if the sum of the singular values of
$\sqrt{A^* A}$ is convergent.
For such an operator, its \textit{trace} is defined to be $\tr(A) :=
\sum_{x \in X} \tangle{e_x, A e_x}$.
%The vector space of trace class operators is denoted by
%$\mathcal{S}_1(\calh)$.

\item Given a vector $u \in \calh$, the corresponding \textit{multiplier}
$M_u : \calh \to \calh$ is given by:
$\tangle{e_x, M_u e_y} := \delta_{x,y} \tangle{e_x, u}$ for all $x, y \in
X$. In other words, $M_u$ is a diagonal operator with respect to the
given basis $\{ e_x \}$, with the corresponding coordinates of the vector
$u$ as its diagonal entries.
\end{enumerate}

Next, we collect together some well-known properties of these operators;
see e.g.~\cite{GK}.

\begin{lemma}\label{Lbasics}
Suppose $(\calh, \tangle{\cdot, \cdot}, \{ e_x : x \in X \})$ is as
above, and $u,v \in \calh$. Also fix another Hilbert space $\calh'$ with
a fixed orthonormal basis $\{ f_y : y \in Y \}$.
\begin{enumerate}
\item The space $\HS(\calh)$ is a two-sided $*$-ideal in
$\mathcal{B}(\calh)$, which contains the multipliers $M_u$.

\item The subspace $\HS(\calh,\calh') \subset \mathcal{B}(\calh, \calh')$
contains all rank-one operators $\lambda u v^* := \lambda u
\tangle{v,\cdot}$ for $\lambda \in \C$, $v \in \calh$, $u \in \calh'$.
Moreover, $* : \HS(\calh, \calh') \to \HS(\calh', \calh)$.

\item If $A \in \HS(\calh, \calh'), B \in \HS(\calh', \calh)$, then $AB,
BA$ are trace class, and their traces coincide.

\item The assignment $(A,B) \mapsto \tr(A^* B)$ is an inner
product on $\HS(\calh, \calh')$.

\item $\HS(\calh)$ is closed under taking Schur products (with respect to
$\{ e_x : x \in X \}$).

\item If $A = \sum_{j=1}^k \lambda_j u_j v_j^*$ is of finite rank for
$u_j, v_j \in \calh$, then $A$ is trace class and $\tr(A) = \sum_{j=1}^k
\lambda_j \tangle{v_j,u_j}$.

\item The multipliers $M_u, u \in \calh$ pairwise commute and are
Hilbert--Schmidt.
\end{enumerate}
\end{lemma}

A simple observation is that Hilbert--Schmidt operators are closed under
composition:

\begin{cor}\label{CHScomp}
Suppose $\calh^{(j)}$ is a Hilbert space with a
fixed orthonormal basis (indexed by) $X^{(j)}$, for $j=1,2,3$. If
$A^{(1)} : \calh^{(1)} \to \calh^{(2)}$ and
$A^{(2)} : \calh^{(2)} \to \calh^{(3)}$ are Hilbert--Schmidt, then so is
their composition $A^{(2)} A^{(1)} : \calh^{(1)} \to \calh^{(3)}$.
\end{cor}

\begin{proof}
Since $A^{(2)}$ is bounded and $A^{(1)}$ is Hilbert--Schmidt, we compute
directly:
\[
\sum_{x \in X^{(1)}} \| A^{(2)} A^{(1)} e_x \|_{\calh^{(3)}}^2 \leq
\| A^{(2)} \|^2 \sum_{x \in X^{(1)}} \| A^{(1)} e_x \|_{\calh^{(2)}}^2
< \infty. \qedhere
\]
\end{proof}

We require a few more notions:

\begin{defn}\label{D23}
Let $\calh, X$ be as above.
\begin{enumerate}
\item Given an operator $A : \calh \to \calh$ and a subset $J \subset X$,
define its `principal submatrix' $A_{J \times J} : \calh \to \calh$ via:
\[
\tangle{e_x, A_{J \times J} e_y} := {\bf 1}_{x \in J} {\bf 1}_{y \in J}
\tangle{e_x, A e_y}, \qquad \forall x,y \in X.
\]
Notice, $A_{J \times J}$ is precisely the compression $P_J A P_J$, where
$P_J$ is the orthogonal projection onto the closed subspace $\calh_J
\subset \calh$ spanned by $\{ e_j : j \in J \}$.

\item For $A \in \HS(\calh)$, define its `diagonal vector' $d_A \in
\calh$ via:
$\tangle{e_x, d_A} := \tangle{e_x, A e_x}$.

\item An operator $A \in \mathcal{B}(\calh)$ is \textit{positive} if $A =
A^*$ (self-adjoint) and $\tangle{u, Au} \geq 0$ for all $u \in \calh$.
\end{enumerate}
\end{defn}

%Finally, we recall the spectral theorem for compact self-adjoint
%operators, by which every finite-rank self-adjoint operator $A$ can be
%written as a finite sum
%\[
%A = \sum_{j=1}^r \lambda_j u_j u_j^* = \sum_{j=1}^r \lambda_j u_j
%\tangle{u_j,\cdot}, \qquad \lambda_j \in \R\ \forall j,
%\]
%with the $u_j$ orthonormal. In particular, if $A$ is positive then one
%can define its square root $\sqrt{A} := \sum_{j=1}^r \sqrt{\lambda_j} u_j
%\tangle{u_j,\cdot}$.

With these preparations, we are ready to extend Theorem~\ref{Tmain0} to
Hilbert--Schmidt operators:

\begin{theorem}\label{Thilbert}
Fix $\calh, X$ as above, and Hilbert spaces $\calh_1, \calh_2$.
Suppose $C_1, C_2 \in \mathcal{S}_2(\calh_1, \calh)$ and $C_3 \in
\mathcal{S}_2(\calh_2, \calh)$,
and define $J := \{ x \in X : C_3^* e_x \neq 0 \}$. If $(C_1 C_1^*)_{J
\times J}, (C_2 C_2^*)_{J \times J}$ are nonzero, then
\begin{equation}\label{Eineq2}
C_3^* (C_1 C_1^* \circ C_2 C_2^*) C_3 \geq \gamma(C_1, C_2, J) \cdot
C_3^* d_{C_1 C_2^T} d^*_{C_1 C_2^T} C_3,
\end{equation}
where $\gamma(C_1, C_2, J)$ is as in~\eqref{Egamma}. Moreover,
the coefficient $\gamma(C_1, C_2, J)$ is best possible.
\end{theorem}

\begin{proof}[Sketch of proof]
If both $(C_1 C_1^*)_{J \times J}$ and $(C_2 C_2^*)_{J \times J}$ have
infinite rank, then (the denominator on) the right-hand side vanishes and
the inequality reduces to the Schur product theorem. It is when at least
one of these ranks is finite that the theorem provides a nonzero lower
bound. 
In this case, one combines the proofs of Theorems~\ref{Tmain0}
and~\ref{Tmain2}; as there are subtleties given the
infinite-dimensionality, we provide some details. First note that if $A
\in \mathcal{B}(\calh',\calh)$ for a Hilbert space $\calh'$ (with $\calh$
as in the theorem), then
\[
C_3^* A A^* C_3 = C_3^* P_J A A^* P_J C_3 = C_3^* (P_J A) (P_J A)^* C_3,
\]
where the orthogonal projection $P_J$ (onto the closed subspace
$\calh_J$) is as in Definition~\ref{D23}(1). Also note that
post-composition by $P_J$ sends the space $\HS(\calh',\calh)$ to
$\HS(\calh',\calh_J)$, and also sends finite-rank operators to
finite-rank operators. Thus, it suffices to prove the theorem without the
$C_3, C_3^*$, and with $C_1, C_2$ replaced by $P_J C_1, P_J C_2 \in
\HS(\calh_1, \calh_J)$ respectively. This essentially reduces the
situation to $J = X$, i.e.~to Theorem~\ref{Tmain0} over $(\calh,X)$ --
here we use that $\rk(T_{J \times J}) \leq \rk(T)$ for all $J \subset X$
and operators $T$ of finite rank.

Thus, we assume henceforth that $J = X$, and repeat the proof of
Theorem~\ref{Tmain0} carefully. First notice by Corollary~\ref{CHScomp}
that $M = C_1 C_1^*, N = C_2 C_2^* \in \HS(\calh,\calh)$, whence so is
$C_1 C_1^* \circ C_2 C_2^*$ by Lemma~\ref{Lbasics}. We now use the key
identity~\eqref{Eidentity0} applied to these $M,N$; firstly, this makes
sense as at least one of $M,N$ is now of finite rank, so that the
right-hand side has finite rank and hence is trace class by
Lemma~\ref{Lbasics}.
Second, the identity~\eqref{Eidentity0} specialized to $M = C_1 C_1^*, N
= C_2 C_2^*$ holds because both sides are additive and continuous in $u,
v \in \calh$ and hence can be reduced to (the easily verifiable case of)
$u = e_x, v = e_y$. Thus we obtain (with $M_u$ in place of $D_u$, and
$v=u$):
\[
\langle u, (C_1 C_1^* \circ C_2 C_2^*) u \rangle = \tr(\overline{C_2}
C_2^T M_{\overline{u}} C_1 C_1^* M_u),
\]
where $\overline{u} \in \calh$ is defined via: $\langle e_x, \overline{u}
\rangle := \langle u, e_x \rangle$. But this equals $\tr (N^* N)$ by
Lemma~\ref{Lbasics}, where $N := C_1^* M_u \overline{C_2}$ now (instead
of $A^* D_u \overline{B}$). This is justified because at least one of
$C_1, C_2$ has finite rank, whence so does $N$; now $N$ is trace class by
Lemma~\ref{Lbasics} (and hence in $\HS(\calh,\calh)$).

For the same reasons, the same properties are satisfied by the projection
$P$ defined as in~\eqref{Eproj} (with $C_1, C_2$ in place of $A,B$). Now
the remainder of the proof of Theorem~\ref{Tmain0} goes through with
minimal modifications.
\end{proof}

\begin{remark}
It is natural to ask if Theorem~\ref{Thilbert} follows from
Theorem~\ref{Tmain} by restricting all operators in question to some
common finite-dimensional space, e.g.~the column space of the matrix on
the left side. However, for infinite $X$ this is not clear, because
such a subspace need not contain a subset of $\{ e_x : x \in X \}$ as a
basis, and our Schur product is with respect to this basis $\{ e_x : x
\in X \}$.
\end{remark}
%}}}

\subsection*{Acknowledgments}

This work is partially supported by
Ramanujan Fellowship grant SB/S2/RJN-121/2017,
MATRICS grant MTR/2017/000295, and
SwarnaJayanti Fellowship grants SB/SJF/2019-20/14 and DST/SJF/MS/2019/3
from SERB and DST (Govt.~of India),
and by grant F.510/25/CAS-II/2018(SAP-I) from UGC (Govt.~of India).
I thank Bhaskar Bagchi, Alexander Belton, Aditya Guha Roy, Gadadhar
Misra, Mihai Putinar, Ajit Iqbal Singh, and Jan Vyb\'iral for valuable
comments and suggestions. Finally, I am grateful to the referee for
carefully going through the manuscript and offering several constructive
comments that helped improve the exposition.

%{{{1 Bibliography

%}}}

%{{{1 Appendix A - Further ramifications
\appendix

\section{Further ramifications}

We provide here a few related but somewhat peripheral observations.

\subsection{Entrywise polynomial preservers in fixed dimension}

The above results reinforce the subtlety of the entrywise calculus. As
observed by P\'olya--Szeg\"o~\cite[Problem 37]{polya-szego}, the Schur
product theorem implies that every convergent power series $f(x)$ with
real non-negative Maclaurin coefficients, when applied entrywise to
positive matrices of all sizes with all entries in the domain of $f$,
preserves matrix positivity. A famous result by Schoenberg~\cite{S3} and
its strengthening by Rudin~\cite{Rudin59} provide the converse for $I =
(-1,1)$: there are no other such positivity preservers. These works have
led to a vast amount of activity on entrywise preservers -- see
e.g.~\cite{BGKP-survey1} for more on this.

If one restricts to matrices of a fixed dimension $n$, the situation is
far more challenging and a complete characterization remains open even
for $n=3$. In this setting, partial results are available when one
restricts the class of test functions, or the class of test matrices in
$\bp_n$ -- see~\cite{BGKP-survey2} for details.

We restrict here to a brief comparison of Vyb\'iral's Theorem~\ref{TV1}
with basic results in our recent work \cite{KT} with Tao and its `baby
case' \cite{BGKP-fixeddim} with Belton--Guillot--Putinar. These latter
two papers study entrywise polynomial maps that preserve positivity on
$\bp_n$ for fixed $n$, and we show in them that for real matrices in
$\bp_n$ with entries in $(0,\epsilon)$ (resp.~$(\epsilon,\infty)$) for
any $\epsilon > 0$, if an entrywise polynomial preserves positivity on
such matrices of rank one, then its first (resp.~last) $n$ nonzero
Maclaurin coefficients must be positive. Contrast this with
Theorem~\ref{Tmain} (or Theorem~\ref{TV1} together with the Schur product
theorem), which shows that for all real correlation matrices in $\bp_n$,
of a fixed dimension $n$, the polynomials $x^{2k} - 1/n, \ k \geq 1$
preserve matrix positivity when applied entrywise.

One hopes that this contrast, together with Remark~\ref{Rlevel} and the
work~\cite{V}, will lead to further new bounds and refined results for
the entrywise calculus on classes of positive matrices.

\begin{remark}
On the topic of the entrywise calculus: notice that if one
applies~\eqref{Eineq0} with $A, B$ to be the positive square roots of $M,
N$ respectively, then
\begin{equation}
M \circ N \geq \frac{1}{\min(\rk(M), \rk(N))} d_{\sqrt{M}
\sqrt{\overline{N}} } d^*_{\sqrt{M} \sqrt{\overline{N}} }, \qquad \forall
M, N \in \bp_n(\C), \ n \geq 1.
\end{equation}
This provides a connection (and a `tight' one) between the entrywise and
functional calculus.
\end{remark}

\subsection{Positive definite functions and related kernels}

As Vyb\'iral remarks in~\cite{V}, if $g$ is any positive definite
function on $\R^d$, or on a locally compact abelian group $G$, then
Theorem~\ref{TV1} immediately implies a sharpening of the `easy half of
Bochner's theorem' for $|g|^2$. We elaborate on this and other
applications through the following unifying notion:

\begin{defn}
Given a set $X$ and a sequence of positive matrices $\mathscr{M} = \{ M_n
\in \bp_n : n \geq 1 \}$, a \textit{complex positive kernel on $X$ with
lower bound $\mathscr{M}$} is any function $K : X \times X \to \C$ such
that for all integers $n \geq 1$ and points $x_1, \dots, x_n \in X$, the
matrix $(K(x_i, x_j))_{i,j=1}^n \geq M_n \geq {\bf 0}_{n \times n}$.
\end{defn}

\noindent Note, positive definite functions/kernels are special cases
with $M_n = {\bf 0}_{n \times n}$. By Theorem~\ref{Tmain}:

\begin{prop}\label{Pkernel}
Suppose $k \geq 1$, and for each $1 \leq j \leq k$, the function $K_j$ is
a complex positive kernel on a set $X_j$, with common lower bound $\{
{\bf 0}_{n \times n} : n \geq 1 \}$. Also suppose $K_j(x_j,x_j) = \ell_j
> 0\ \forall x_j \in X_j, \ 1 \leq j \leq k$. Then the kernel ${\bf K}$
on $X_1 \times \cdots \times X_k$ given by
\[
{\bf K}((x_1, \dots, x_k), (x'_1, \dots, x'_k)) := \prod_{j=1}^k K_j(x_j,
x'_j) K_j(x'_j, x_j), \qquad x_j, x'_j \in X_j,
\]
is complex positive on $X_1 \times \cdots \times X_k$ with lower bound
$\{ \frac{1}{n} \prod_{j=1}^k \ell_j \cdot E_n : n \geq 1 \}$.
\end{prop}

This setting and result unify several different notions in the
literature, as we now explain:
\begin{enumerate}
\item \textit{Positive definite functions on groups:}
Here $X$ is a group with identity $e_X$, and $K$ is the composite of the
map $(x,x') \mapsto x^{-1} x'$ and a function $g : X \to \C$ satisfying:
$g(x^{-1}) = \overline{g(x)}$. Then the hypotheses of
Proposition~\ref{Pkernel} apply in this case, with $\ell := g(e_X)$.

For instance, in~\cite{V} the author uses the positive definiteness of
the cosine function\footnote{On a related note: Vyb\'iral mentions
in~\cite{V} that $\cos(\cdot)$ is positive definite on $\R^1$ using
Bochner's theorem. A simpler way to see this uses trigonometry: given
reals $x_1, \dots, x_n$, the matrix $(\cos(x_i - x_j))_{i,j=1}^n = u u^T
+ v v^T$, where $u = (\cos x_j)_{j=1}^n$ and $v = (\sin x_j)_{j=1}^n$.}
on $\R$ to apply Theorem~\ref{TV1} and prove a conjecture of
Novak~\cite{N} -- see Theorem~\ref{Tnovak} below. This now follows from
Proposition~\ref{Pkernel} -- we present here a more general version than
in~\cite{V}:

\begin{prop}
Let $\mu_1, \dots, \mu_k$ be finite non-negative Borel measures on $X$,
and $g_l$ the Fourier transform of $\mu_l$ for all $l$. Then,
\[
(\prod_{l=1}^k |g_l(x_i^{-1} x_j)|^2)_{i,j=1}^n \geq \frac{1}{n}
\prod_{l=1}^k g_l(e_X)^2 \cdot E_n.
\]
\end{prop}

\item \textit{Positive semidefinite kernels on Hilbert spaces:}
Here $(X, \tangle{\cdot, \cdot})$ is a Hilbert space over $\R$ or $\C$,
and $K$ is the composite of the map $(x,x') \mapsto \tangle{x,x'}$ and a
function $g : \C \to \C$ satisfying: $g(\overline{z}) = \overline{g(z)}$.
(See e.g.~the early work by Rudin~\cite{Rudin59}, which classified the
positive semidefinite kernels on $\R^d$ for $d \geq 3$, and related this
to harmonic analysis and to the entrywise calculus.) In this case
Theorem~\ref{Tmain} applies; if one restricts to kernels that are
positive definite on the unit sphere in $X$, then
Proposition~\ref{Pkernel} applies here as well, with $\ell :=
g(1)$ -- and thus applies to \textit{covariance kernels}, widely used in
the (statistics) literature.\medskip

\item \textit{Positive definite functions on metric spaces:}
In this case, $(X,d)$ is a metric space, and $K$ is the composite of the
map $(x,x') \mapsto d(x,x')$ and a function $g : [0,\infty) \to \R$. This
was studied by several experts including Bochner, Weil, and Schoenberg.
For instance, Schoenberg observed in~\cite{S1} that $\cos(\cdot)$ is
positive definite on unit spheres in Euclidean spaces, and went on to
classify in~\cite{S3} the positive definite functions $f \circ \cos$ on
spheres of each fixed dimension $d$. The $d=\infty$ case is the
aforementioned `converse' to the Schur product theorem (i.e., it shows
that the P\'olya--Szeg\"o observation above is `sharp').

We conclude with a specific example, which leads to another result
similar to Novak's conjecture (shown by Vyb\'iral). A well-known
result of Schoenberg~\cite{S2} says that the Gaussian kernel
$\exp(-\lambda x^2)$ is positive definite on Euclidean space for all
$\lambda > 0$.\footnote{On a related note: Schoenberg~\cite{S2} shows the
positive definiteness of the Gaussian kernel using Fourier analysis. In
the spirit of the preceding footnote, we provide a purely
matrix-theoretic proof in three steps -- we also include this in the
recent survey~\cite{BGKP-survey1}:\hfill\break
(1)~A result of Gantmacher--Krein says square generalized Vandermonde
matrices $(x_j^{\alpha_k})$ have positive determinant if $0 < x_1 < x_2 <
\cdots$ and $\alpha_1 < \alpha_2 < \cdots$ are real.\hfill\break
(2)~This implies an observation of P\'olya: the Gaussian kernel is
positive definite on $\R^1$. Indeed, given $x_1 < x_2 < \cdots$, the
matrix $(\exp(-(x_j - x_k)^2))$ equals $DVD$, where $D$ is the diagonal
matrix with diagonal entries $\exp(-x_j^2)$, and $V =
(\exp(2x_j)^{x_k})$ is a generalized Vandermonde matrix.\hfill\break
(3)~The positivity of the Gaussian kernel on every Euclidean space
$\R^d$, whence on Hilbert space $\ell^2(\mathbb{N})$, now follows from
P\'olya's observation via the Schur product theorem.}
(In fact Schoenberg shows this characterizes Hilbert space
$\ell^2(\mathbb{N})$, i.e.~the completion of $\bigcup_{d \geq 1} (\R^d,
\| \cdot \|_2)$.) Thus:

\begin{prop}
Given $x_{l1}, \dots, x_{ln} \in \ell^2(\mathbb{N})$ for $l=1,\dots,k$,
the $n \times n$ real matrix with $(i,j)$ entry
$\prod_{l=1}^k \exp(-\| x_{li} - x_{lj} \|^2) - \frac{1}{n}$
is positive semidefinite.
\end{prop}

This is similar to Novak's conjecture, now shown by Vyb\'iral:

\begin{theorem}[\cite{N,V}]\label{Tnovak}
Given $x_{l1}, \dots, x_{ln} \in \R$ for $l=1,\dots,k$,
the $n \times n$ real matrix with $(i,j)$ entry
$\prod_{l=1}^k \cos^2(x_{li} - x_{lj}) - \frac{1}{n}$
is positive semidefinite.
\end{theorem}

The two results are similar in that Novak's conjecture uses $\cos(\cdot)$
and $\R^1$ in place of $\exp(-(\cdot)^2)$ and $\ell^2(\mathbb{N})$
respectively. Both results follow from Proposition~\ref{Pkernel}.
\end{enumerate}
%}}}

\end{document}